\documentclass[11pt]{article}
  \usepackage{times}
  \usepackage{mathdots}
  \usepackage{amssymb,amscd,amsthm,latexsym}
  \usepackage{amsmath}
  \usepackage{graphicx}
  \usepackage{subcaption}	
  \usepackage{float}
  \usepackage{amssymb}
  \usepackage{epstopdf}
  \usepackage{verbatim}
  \newtheorem{theorem}{Theorem}[section]
  \newtheorem{proposition}[theorem]{Proposition}

  \theoremstyle{definition}
  \newtheorem{definition}[theorem]{Definition}
  \newtheorem{example}[theorem]{Example}
  
  \newtheorem{remark}[theorem]{Remark}

  \newtheorem{cor}[theorem]{Corollary}

  \newtheorem{rem}[theorem]{Remark}

  \def\gs{Gershgorin disc of the second type\,}
  
  \def\gss{Gershgorin discs of the second type\,}
  \def\grs{Gershgorin region of the second type\,}
  
  \def\a{\alpha~}
  
  \def\l{\lambda}
  \def\n{n \times n~}

  \def\r{\hat{\rho}}
  \title{\vskip10mm Inclusion regions and bounds for the eigenvalues of matrices with a known eigenpair}
  \author{
  Rachid Marsli\footnote{ Supported by grant no. SB181014 of King Fahd University of Petroleum and Minerals.}
  \\ Preparatory Math Department\\
  King Fahd University of Petroleum and Minerals\\Dhahran, 31261\\
  Kingdom of Saudi Arabia\\ rmarsliz@kfupm.edu.sa\\
  Frank J. Hall \\ Department of Mathematics and Statistics\\ Georgia State University\\ Atlanta, GA 30303, USA\\ fhall@gsu.edu
  	}
  \date{\today}
\begin{document}
  %
  %
  %
  \maketitle
  \begin{abstract}
  	 Let $(\lambda, v)$ be a known real eigenpair of an $n \times n$ real matrix $A$. In this paper it is shown how to locate the other eigenvalues of $A$ in terms of the components of $v$. The obtained region is a union of Gershgorin discs of the second type recently introduced by the authors in a previous paper. Two cases are considered depending on whether or not some of the components of $v$ are equal to zero.
  	 Upper bounds are obtained, in two different ways, for the largest eigenvalue in absolute value of $A$ other than $\l$.  	 
  	 Detailed examples are provided.
  	 Although nonnegative irreducible matrices are somewhat emphasized, the main results in this paper are valid for any $\n$ real matrix with $n\geq3$.   	      	 
  \end{abstract}
  \noindent
  {\it AMS Subj. Class.:} 15A18; 15A42
  \section{Introduction}
  In our recent work \cite{HM2} and \cite{HM3}, we showed how to locate the eigenvalues of any real constant row-sum matrix within a region that is smaller than  or, in the worst case, equal to the traditional Gershgorin region. What helped us to obtain such results is the fact that a constant row-sum matrix has a trivial eigenvalue equal to its constant row-sum which is associated with a known eigenvector, the all $1$'s vector.
  If a real matrix $A$ is nonnegative and irreducible, its Perron eigenvector $v_p$ is known to have all positive components. A goal of this work is to show how to use this information to locate the non-Perron eigenvalues of $A$ in terms of the components of $v_p$.\\
  The idea of using a known eigenpair to locate the spectrum of a given matrix was explored recently in the interesting paper \cite{M} by A. Melman.
  Let $(\lambda, v)$ be a known real eigenpair of an $n \times n$ real matrix $A$. In this paper it is shown how to locate the other eigenvalues of $A$ in terms of the components of $v$. The obtained region is a union of Gershgorin discs of the second type recently introduced by the authors in a previous paper. Two cases are considered depending on whether or not some of the components of $v$ are equal to zero.
  Upper bounds are obtained in different ways for the largest eigenvalue in absolute value of $A$ other than $\l$.  	 
  Detailed examples are provided. In Sections 2, 3 and 4, comparisons are made with the original Gershgorin region. In Sections $5$ and $6$, comparisons are made with the location sets and bounds obtained in \cite{M}.
  Although nonnegative irreducible matrices are somewhat emphasized, the main results in this paper are valid for any $\n$ real matrix with $n\geq3$.  
  \section{ First case: $v$ has no zero components} 
  		If $A$ is a real constant row-sum or  a real constant column sum matrix, then a way to obtain an inclusion region for its eigenvalues is described in \cite{HM2}.
  		The obtained region is a union of the Gershgorin discs of the second type which were defined for the first time in \cite{HM1}. Further results on how to improve the location of eigenvalues of real constant row-sum and constant column-sum matrices are obtained in \cite{HM3}.
  		We assume in this section that $A$ is neither constant row-sum nor constant column sum; however, our results also particularly apply to these matrices.\\
  		If $A$ is nonnegative and irreducible, we know that it has a positive eigenvalue $\l_p$ equal to its spectral radius, \cite[Theorem 8.4.4]{HJ}. This eigenvalue, called the Perron eigenvalue of $A$, is simple and is associated with an eigenvector $v_p$
  		called the Perron eigenvector of $A$ and having the property that all its components are strictly positive. 
  		We show how to use $v_p$ to obtain,  for the non-Perron eigenvalues of $A$, an inclusion region that is a union of discs different from those given by the original Gershgorin theorem. To do this, we use the following theorem that allows us to convert $A$ to a matrix $B$ that is constant row-sum and at the same time similar to $A$; the case where $A$ is nonnegative irreducible can be found in \cite{BDS}.
 	\begin{theorem}\label{t1}
 		Let $A$ be an $n \times n$ real matrix.
 		Let $(\l,v)$ be a real eigenpair of $A$ with $v = (v_1, v_2, \dots, v_n)^T$.
 		Suppose that $v_i \neq 0$ for $i=1, 2, \dots, n.$
 		Let $S = \text{diag}(v_1, v_2, \dots, v_n)$ and $B = [b_{ij}]$ be the matrix similar to $A$ given by
 		\begin{equation}\label{e1}
 		B = S^{-1} A S.
 		\end{equation}
 		Then $B$ is a constant row-sum matrix with $Be = \l e$, where $e$ is the all $1$'s vector in $\mathbb{R}^n$.
 	\end{theorem} 
    \begin{proof}
 	Straightforward by calculation of  $Be$.
     \end{proof}
 	\begin{rem}
 		The elements of $B$ are given by
 		\begin{equation}\label{e2}
 		b_{ij} = v_i^{-1} a_{ij} v_j.
 		\end{equation}
 		Note that the  elements of $B$ remain constant if we scale the vector $v$ by any positive number $\a$ since
 		$$(\a\,v_i)^{-1} a_{ij} (\a v_j) = v_i^{-1} a_{ij} v_j = b_{ij}.$$
 		Note also that the expression $Be = \l e$ means that the row-sum of every row of $B$ is equal to $\l$. As in our previous work, the eigenvalue $\l$ is called the trivial eigenvalue of $B$ while any eigenvalue of $B$ that is different from  $\l$ is called a non-trivial eigenvalue of $B$. 	
 	\end{rem}
 
 		\noindent
  		   Since $B$ is a constant row-sum matrix,  we can locate its non-trivial eigenvalues by applying the following theorem
  	       which can be found in our paper \cite[Theorem 2.2]{HM2}.
  	     \begin{theorem}\label{t2}
  	     Let $B$ be an $n \times n$ real constant row-sum matrix.
  	     All the non-trivial eigenvalues of $B$ are located within the union of the Gershgorin discs of the second type of $B^T$.
         \end{theorem}
     
     	\noindent
        	This type of Gershgorin disc was introduced in \cite{HM1}.
        \begin{definition}\cite[Definition 2.5]{HM1}\label{d1}
        	Let $A = [a_{ij}]$ be an $n \times n$ real matrix.
        	For $\ \ i = 1 , ... , n,\,$ let $x_{i 1} \geq   \, ... \,  \geq  x_{i n}\,$ be a rearrangement in non-increasing order
        	of $$\, a_{i 1}\, , ... , \, a_{i i-1} \, , \, 0 \, , \, a_{i i+1}\, , ... , \, a_{i n}.$$
        	A Gershgorin disc of the second type of $A$, denoted $\hat{D}_i(a_{ii} \, , \, \hat{r}_i)$, satisfies the following conditions:
        	\begin{enumerate}
        		\item Its center $\,a_{ii}\,$ is the diagonal element from the $i^{th}$ row of $A$
        		\item Its radius: 
        		\begin{enumerate}
        			\item $  \hat{r}_i \ \  = \ \ \displaystyle{\sum_{j=1}^{\frac{n-1}{2}} \, x_{i j} \, - \, \sum_{j=\frac{n+3}{2}}^n \, x_{i j}}  \ \ ,\ \ $ if $\,n\,$ is odd.
        			\item $ \hat{r}_i \ \  = \ \ \displaystyle{\sum_{j=1}^{\frac{n}{2}} \, x_{i j} \, - \, \sum_{j=\frac{n}{2}+1}^n \, x_{i j}}  \ \ ,\ \ $ if $\,n\,$ is even.
        		\end{enumerate}
        	\end{enumerate}
        	The \grs~of $A$ is the union of all its \gss.
        \end{definition}
    		
    		\noindent
       		Going back to Theorem \ref{t1}, the matrix $B$ is obtained from the real matrix $A = [a_{ij}]$ according to (\ref{e1}) and has the same spectrum as $A$.
       		Therefore, the combination of Theorem \ref{t1} and Theorem \ref{t2} gives an inclusion region for all the eigenvalues of $A$ other than $\l$.
       		Let's state this result as a theorem.
        \begin{theorem}\label{t3}
       		Let $A$ be a real matrix and suppose that $A v = \l v$ for some real number $\l$ and real eigenvector
       	 	$v = (v_1, v_2, \dots, v_n)^T $ with no zero component.
       		Let $S = \text{diag}(v)$ and $B=S^{-1} A S$.
       		If $\l_2$ is an eigenvalue of $A$ different from $\l$, then $\l_2$ is in the Gershgorin region of the second type of $B^T$.
        \end{theorem}
   		\begin{rem}
   			Observe that $$b_{jj} = \frac{v_j}{v_j}a_{jj} = a_{jj}.$$
   			This is important as it means that the components of $v$ are not needed to find the centers of discs which are simply the diagonal elements of $A$.
   			The radii depend on the entries of $v$, but they are easy to calculate as each one of them is a simple algebraic sum of the elements from the corresponding column of $B$ (according to Definition \ref{d1}).	
   		\end{rem}   
   	
   		\noindent		
   			The Perron eigenvector of every nonnegative irreducible matrix is positive. Therefore we have the following corollary.   		 
   		\begin{cor}\label{c1}
   				Let $A$ be an $n\times n$  nonnegative  irreducible matrix and let $v_p$ be its Perron eigenvector.
   				Then every non-Perron eigenvalue of $A$ is in the Gershgorin region of the second type of  the  matrix $B^T$ given by
   				$B = S^{-1} A S,$ where $S= \text{diag}(v_p)$. 
   		\end{cor}
       	\begin{example}\label{ex1}
       		In this example, we provide a matrix $A$ and its Perron eigenpair; then we show how the other eigenvalues can be located according to Corollary \ref{c1}.
       		Let  $$A\,=\, \left[ \begin{array}{r r r r r r r r}
       	    10 & 4 &  8 &  4 & 6 &  6 \\
       	    2 & 6 &  6 &  2 & 4 &  2 \\
       	    1 & 4 &  8 &  4 & 2 &  4 \\
       	    0 & 6 &  8 &  4 & 0 &  6 \\
       	    4 & 4 &  6 &  0 & 2 &  4 \\
       	    1 & 4 &  6 &  2 & 4 &  6 
           	\end{array} \right].$$
           	The Perron eigenvalue of $A$ is $\l_p = 24$ and its Perron eigenvector is $v_p = (2, 1, 1, 1, 1, 1)^T$.
           	Let $$S = \text{diag}(v_p)$$ and let
           	$$B\,=\,S^{-1} A S \,=\, \left[ \begin{array}{r r r r r r r r}
           	10 & 2 &  4 &  2 & 3 &  3 \\
           	4  & 6 &  6 &  2 & 4 &  2 \\
           	2  & 4 &  8 &  4 & 2 &  4 \\
           	0  & 6 &  8 &  4 & 0 &  6 \\
           	8  & 4 &  6 &  0 & 2 &  4 \\
           	2  & 4 &  6 &  2 & 4 &  6 
           	\end{array} \right].$$
           	According to Corollary \ref{c1}, the non-Perron eigenvalues of $A$ are contained in the union of the Gershgorin discs of the second type of $B^T$ which are:
           	$$ \hat{D}_1(10 ,12) ~,~ \hat{D}_2(6,8) ~,~ \hat{D}_3(8,10) ~,~ \hat{D}_4(4,6) ~,~ \hat{D}_5(2,9)  \text{~~and~~} \hat{D}_6(6,9).$$
           	It can be easily verified that the union of the above discs forms an inclusion set for the non-Perron eigenvalues of $A$
           	which are 
           	$$\l_2 \approx 7.76,~ \l_3 \approx -3.05,~ \l_4 \approx 2.65+1.34i,~ \l_5 \approx 2.65-1.34i, \text{~and~}	 \l_6 =2.$$
           	This region is neatly smaller than the usual Gershgorin region of $A^T$ made up of the discs
           	$$ D_1(10 ,8) ~,~ D_2(6,22) ~,~ D_3(8,34) ~,~ D_4(4,12) ~,~ D_5(2,16)  \text{~~and~~} D_6(6,22).$$
           	It is also smaller than the  Gershgorin region of $A$ made up of the discs
           	$$ D'_1(10 ,28) ~,~ D'_2(6,16) ~,~ D'_3(8,15) ~,~ D'_4(4,20) ~,~ D'_5(2,18)  \text{~~and~~} D'_6(6,17).$$
           	This difference in size between these regions is illustrated in the figure below.
           	\begin{figure}[H]
           		\centering
           		\begin{subfigure}[b]{0.35\textwidth} 
           			\includegraphics[scale=0.34]{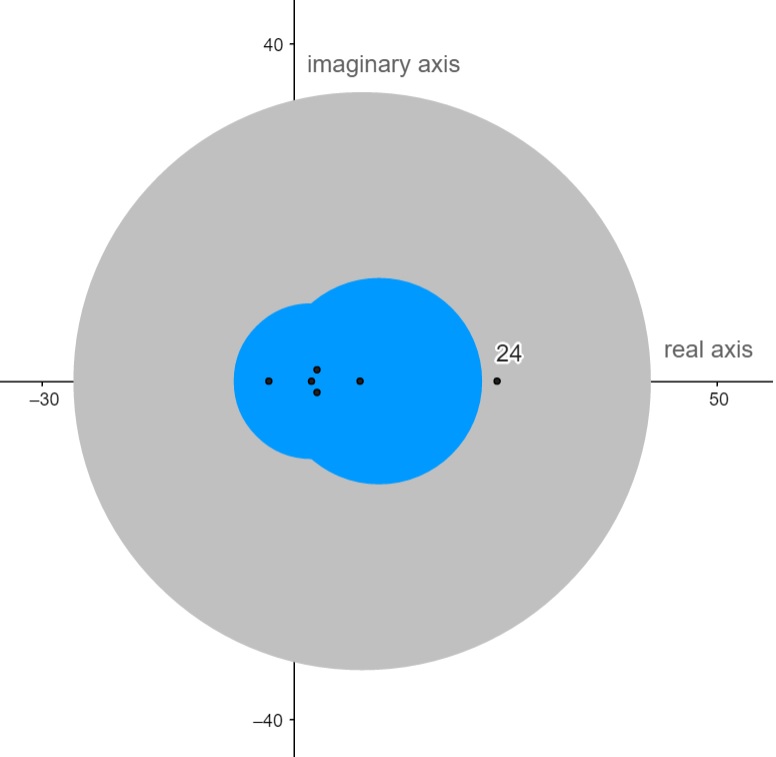}           			          			        
           		\end{subfigure}
           		\hskip20mm
           		\begin{subfigure}[b]{0.30\textwidth}
           			\centering           			
           			\includegraphics[scale=0.30]{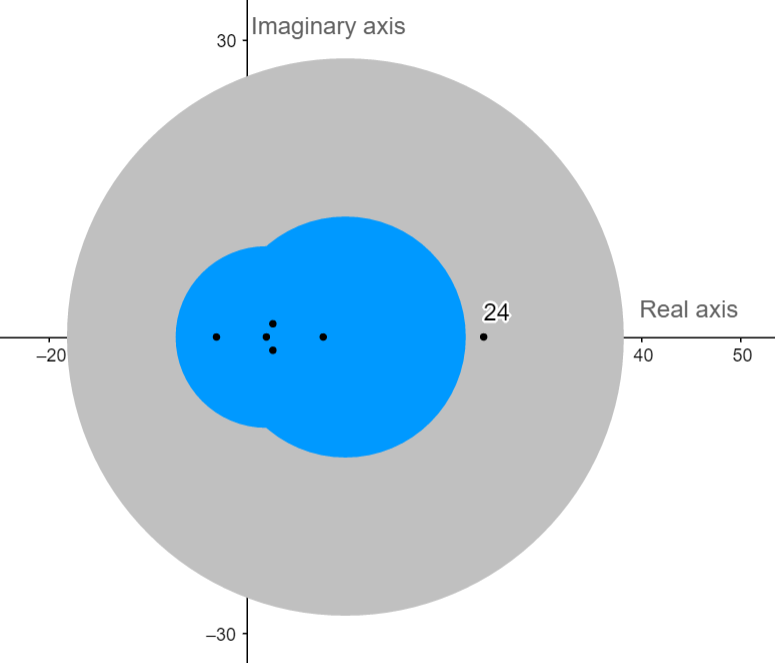}
           			$\vdots$           		
           		\end{subfigure}  
           		\caption{The region given by Corollary \ref{c1} is colored in blue. In the left hand side graph it is compared to the  Gershgorin region of $A^T$ and in the right hand side graph it is compared to the Gershgorin region of $A$. The eigenvalues of $A$ are represented by the small dark points. Observe that all the non-Perron eigenvalues of $A$ are inside the blue region. This is in accordance with Corollary \ref{c1}.
           			Contrary to the Gershgorin theorem, Theorem \ref{t3} and Corollary \ref{c1} impose no constraints on the location of the Perron eigenvalue $\l_p=24$. This explains why $\l_p$ is outside the  blue region for this particular matrix $A$.}		
           	\end{figure}           
            \end{example}
            \begin{rem}
            	Although nonnegative irreducible matrices are emphasized, the results in this section are valid for any $\n$ real matrix having an eigenvector with no zero entry. This is clearly understood from the statement of Theorem \ref{t3}.
            \end{rem}
            \begin{example}\label{ex3}
            	Consider the singular  matrix 
            	$$
            	A\,=\ \left[ \begin{array}{r r r r r r }
            	-2  & 4  &  ~~2  &   0 & -2  &  2   \\
            	-2  & 2  &  ~~2  &  -2 &  0  & -2   \\
            	-4  & 2  &  ~~4  &   0 & -2  &  0   \\
            	-4  & 10 &  ~~6  &  -4 & -16 & -12  \\
            	0  & 4  &  ~~8  &  -4 &  -6 & -2   \\
            	-8  & 2  &  ~~2  &  -2 &   0 & -8
            	\end{array} \right] ,$$
            	with eigenvector $v = (1,\, 1,\, 1,\, 2,\, 1,\, -1)^T$ corresponding to the eigenvalue $\l = 0$.
            	The other eigenvalues of $A$ are approximately $-13.32, ~ -3.71\pm 4.39i~$ and $~3.37\pm 2.12i$.
            	The matrix $B$ given by Theorem \ref{t3} is
            	$$
            	B\,=\ \left[ \begin{array}{r r r r r r }
            	-2  & ~~4  &  ~~2  &    0 &  -2  & -2    \\
            	-2  & ~~2  &  ~~2  &   -4 &   0  &  2    \\
            	-4  & ~~2  &  ~~4  &    0 &  -2  &  0    \\
            	-2  & ~~5  &  ~~3  &   -4 &  -8  &  6    \\
              	 0  & ~~4  &  ~~8  &   -8 &  -6  &  2    \\
            	 8  &  -2  &   -2  &    4 &   0  & -8   
            	\end{array} \right] .$$
            	According to the same theorem, all the nonzero eigenvalues of $A$ are in the union of the \gss of $B^T$ which are\\
            	$\hat{D_1}(-2,16), ~ \hat{D_2}(2,13), ~ \hat{D_3}(4,13), ~ \hat{D_4}(-4,16), ~ \hat{D_5}(-6,12)~$ and $~\hat{D_6}(-8,12).$
            	In the following figure, we compare this region to the Gershgorin region of the matrix $A^T$ made up of the discs\\
            	$D_1(-2,18), ~ D_2(2,22), ~ D_3(4,20), ~ D_4(-4,8), ~ D_5(-6,20)~$ and $~D_6(-8,18).$
            	\begin{figure}[H]
            		\centering
            		\includegraphics[scale=0.40]{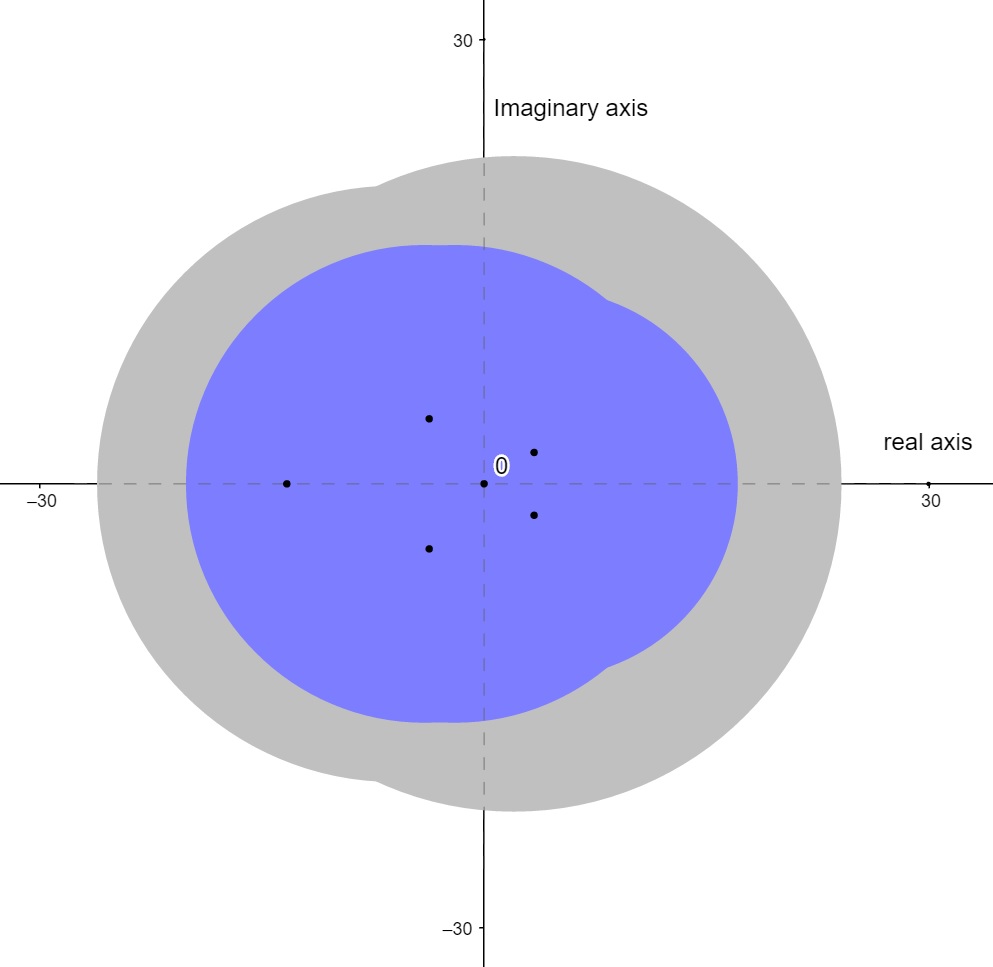}       		

            	\caption{The area in  blue is the \grs of $B^T$.
            		The original Gershgorin region of $A^T$ extends over the gray and  blue areas.
            		The points in black represent the eigenvalues of $A$.
            		Observe that the \grs of $B^T$ is a subset of the Gershgorin region of $A^T$ and all the eigenvalues of $A$ other than $\l=0$ belong to it. This is in accordance with Theorem \ref{t3}. The eigenvalue $\l=0$ also lies within this region for this particular matrix $A$, but this is not a consequence of Theorem \ref{t3}.}         		
            	\end{figure}
            \end{example}
        
        	\noindent
            This generalizes to any singular real matrix with an eigenvector with no zero entry corresponding to the eigenvalue $0$.\\
            
            \noindent
            In the  previous two examples, the \grs of $B^T$ is a subset of the Gershgorin region of $A^T$. Despite the fact that this is what we expect for most matrices, this is not always the case as shown by the following example.
            
            \begin{example}\label{ex11}
            	Let
            		$$
            	A\,=\ \left[ \begin{array}{r r r}
            	9   & 1  &  1      \\
            	0   & 5  &  5      \\
            	4   & 1  &  1                      	 
            	\end{array} \right].$$
            	This matrix has an eigenvector $v = (2,\, 1,\, 1)^T$ which we use to construct the following  matrix $B$ according to Corollary \ref{c1}:
            	$$
            	B\,=\ \left[ \begin{array}{r r r}
            	9   & 0.5   &  0.5     \\
            	0   & ~~5   &  ~~5     \\
            	8   & ~~1   &  ~~1                
            	\end{array} \right].$$           
            	It is easy to check that the \grs of $B^T$  is larger than the Gershgorin region of $A^T$.
            	\begin{figure}[H]
            		\centering
            		\includegraphics[scale=0.40]{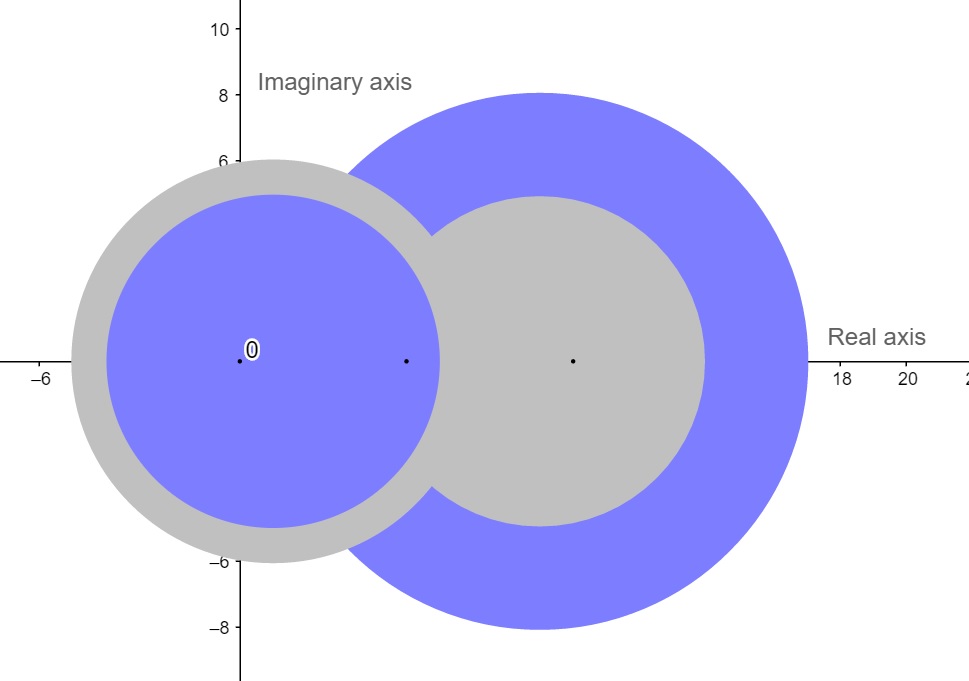}       		
            		\caption{The \grs of $B^T$ is the union of the two blue discs; 
            			the Gershgorin region of $A^T$ is the union of the two gray discs. 
            			From the figure it is clear that the \grs of $B^T$  is larger than the Gershgorin region of $A^T$.
            			Since neither of the regions is a subset of the other, their intersection provides a better inclusion set for the eigenvalues of  $A$ namely $0, 5$ and $10$ which are represented by the small dark points.} 	
            	\end{figure}
            \end{example}
        
        	\noindent
        	In general, we expect that the \grs of the matrix $B^T$ obtained by applying Corollary \ref{c1} to a large dense nonnegative irreducible $\n$ matrix $A$ is smaller than and contained within the Gershgorin region of $A^T$ for the following reason:
        	the radius of a Gershgorin disc is obtained by summing the absolute values of all off-diagonal elements in a given row or column, while the radius of a Gershgorin disc of the second type is obtained by taking the difference between the sum of the largest $\approx \frac{n}{2}$  and the sum of the smallest $\approx \frac{n}{2}$  off-diagonal elements in a given row or column (see Definition \ref{d1}).
        	In particular, if the eigenvector being used in generating the matrix $B$ is relatively flat (the components of $v$ are close to each other), then the chance is even bigger for the inclusion set given by Corollary \ref{c1} to be significantly smaller than and contained within the Gershgorin region of $A^T$.
        	Another reason for which we are interested to the inclusion regions given by Theorem \ref{t3} and Corollary \ref{c1} is that 
        	they can be improved further by applying some ideas from \cite{HM3}. This is what we shall discuss in the next section.
        	%
    	\section{Further refinement of the inclusion regions}
           Let $A$ be $\n$ real matrix with an eigenvector $v$ with no zero component.
           The inclusion region given by Theorem \ref{t3} can be improved further by applying Theorem 3.10 and Theorem 3.15 from \cite{HM3}. The first one applies to constant row-sum matrices of even size and implies the following ideas.
           We note that this statement incorporates a small official journal correction to the original version.
       	\begin{theorem}\label{t4}
       	   Let n be an even integer and
       	   let $B = [b_{ij}]$ be an $n \times n$ real constant row-sum matrix.
       	   Let  $\l_1$ be its constant row-sum and let $e$ be the all $1$'s vector of $n$ components. 
       	   Let $L_j$ be the $j^{th}$ column of $B$ for $j=1, \dots, n$.
       	   Construct the matrix $F = [f_{ij}] = [ L_1 - \beta_1 e, \dots, L_n-\beta_n e]$,
       	   where $\beta_j$ is the $(\frac{n}{2})^{th}$ largest element among
       	   $b_{1j}, \dots, b_{j-1,j} , b_{j+1,j}, \dots, b_{nj}$.
       	   Then the  Gershgorin region of the second type of $F^T$ is a subset of that of $B^T$ and contains every eigenvalue of $B$ different from $\l_1$.       		      	
       \end{theorem}
   		\begin{proof}
   			Look at \cite[Theorem 3.10]{HM3} and its proof.
   		\end{proof}
       \begin{example}\label{ex2}
       	   Let $A$ be as in Example \ref{ex1}.
       	   The constant row-sum matrix obtained from $A$ by application of Corollary \ref{c1} is
       	   $$B\,=\, \left[ \begin{array}{r r r r r r r r}
       	   10 & 2 &  4 &  2 & 3 &  3 \\
       	   4  & 6 &  6 &  2 & 4 &  2 \\
       	   2  & 4 &  8 &  4 & 2 &  4 \\
       	   0  & 6 &  8 &  4 & 0 &  6 \\
       	   8  & 4 &  6 &  0 & 2 &  4 \\
       	   2  & 4 &  6 &  2 & 4 &  6 
       	   \end{array} \right].$$
       	   Now we apply  Theorem \ref{t4} to $B$ to obtain the matrix F:
       	   $$F\,=\ \left[ \begin{array}{r r r r r r r r}
       	   8  & -2 &  -2 &  0 &  0 &  -1 \\
       	   2  &  2 &   0 &  0 &  1 &  -2 \\
       	   0  &  0 &   2 &  2 & -1 &   0 \\
          -2  &  2 &   2 &  2 & -3 &   2 \\
       	   6  &  0 &   0 & -2 & -1 &   0 \\
       	   0  &  0 &   0 &  0 &  1 &   2        	
       	  \end{array} \right].$$
       	Since $A$ and $B$ are similar to each other, it follows by Theorem \ref{t4} that the non-Perron eigenvalues of $A$ are contained in the Gershgorin region of the second type of $F^T$ which is made up of the discs:
       	$$ \hat{D}_1(8 ,10) ~,~ \hat{D}_2(2,4) ~,~ \hat{D}_3(2,4) ~,~ \hat{D}_4(2,4) ~,~ \hat{D}_5(-1,6) \text{~~and~~} \hat{D}_6(2,5).$$
       	\begin{figure}[H]
       		\centering
       	\includegraphics[scale=0.45]{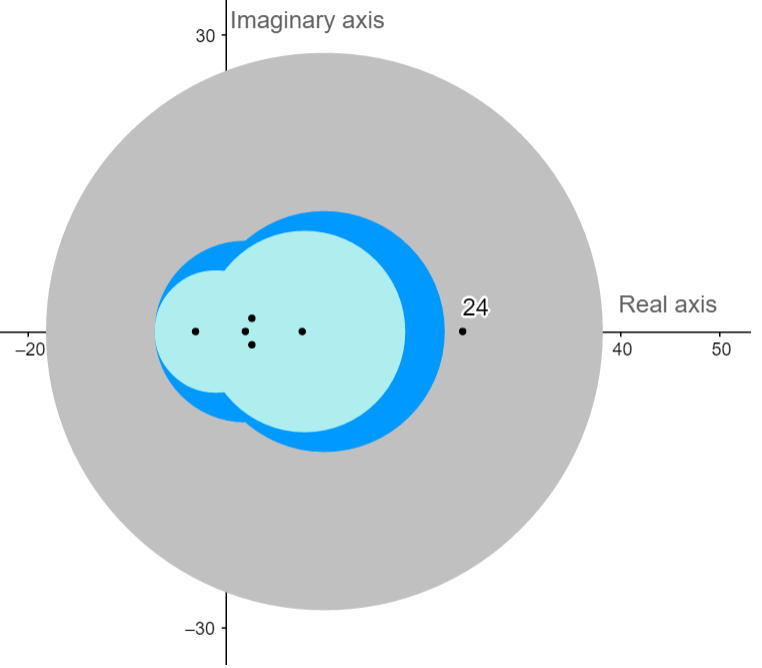}       		
       		\caption{The Gershgorin region of $A^T$ as well as the region given by Corollary \ref{c1} are as in the previous figure.
       			The Gershgorin region of the second type of $F^T$ is represented by the turquoise color. Observe that all the non-Perron eigenvalues of $A$ are contained within this region which is a subset of the   blue region.}
       	\end{figure}       	
       \end{example}
   
   		\noindent
      In the case of matrices of odd size, the region given by Theorem \ref{t3} can be enhanced by using the following theorem and corollary which are adapted versions of \cite[Theorem 3.15 and Corollary 3.16]{HM3}. 
       \begin{theorem}\label{t5}
       	Let n be an odd integer with $n\geq3$,
       	let $B = [b_{ij}]$ be an $n \times n$ real constant row-sum matrix and let  $\l_1$ be its constant row-sum.
       	Let $\beta_j$ and  $\gamma_j$ be, respectively, the opposite in sign of the $(\frac{n-1}{2})^{th}$
       	and the $(\frac{n+1}{2})^{th}$ largest numbers among $ b_{1 j} , \dots , b_{j-1 j} , b_{j+1 j} , \dots , b_{n j}$.
       	Construct the matrices $F = [f_{ij} = b_{ij}+\beta_j]$ and $G = [g_{i j}=b_{ij}+\gamma_j]$,
       	and let 
       	\begin{equation}\label{e3}
       	S = \underset{1\leq j \leq n}{{\bigcup}} \Big(\hat{D}_{F,j} \bigcap \hat{D}_{G,j}  \Big),
       	\end{equation}       	
       	where $\hat{D}_{F,j}$  and $\hat{D}_{G,j}$ are respectively, the Gershgorin discs of the second type
       	obtained from the $j^{th}$ columns of $F$ and $G$.
       	Then the region $S$ is contained within the Gershgorin region of the second type of $B^T$ and if $\l$ is an eigenvalue of $B$ different from $\l_1$, then $\l~\in~S$.       		
       \end{theorem}
       \begin{proof}
       	Look at \cite[Theorem 3.15]{HM3} and its proof.
       \end{proof}
   		\begin{cor}\label{c2}
   			All the eigenvalues other than the trivial eigenvalue of $B$ lie in the intersection of the Gershgorin regions of the second type of $F^T$ and $G^T$.
   		\end{cor}
   		\begin{proof}   	 		
   			The region given by (\ref{e3}) is a subset of this intersection. Look at \cite[Theorem 3.15 and Corollary 3.16 ]{HM3} and their proofs.
   		\end{proof}
   	
   		\noindent
   		The region given by Corollary \ref{c2} is larger than or equal the one given by Theorem \ref{t5}. However it is considered for its relatively simple form. Its graph can by done by graphing the entire regions of $F^T$ and $G^T$, then taking their intersection.
       \begin{example}\label{ex9}
       	Let $$A\,=\, \left[ \begin{array}{r r r r r r r r}
       	2  & 3 &  6 &  9 & 6 &  6  & 6 \\
       	2  & 2 &  4 &  0 & 4 &  6  & 6 \\
       	0  & 1 &  3 &  2 & 4 &  2  & 2 \\
       	2  & 1 &  2 &  0 & 2 &  1  & 2 \\
       	1  & 2 &  1 &  3 & 0 &  3  & 1 \\
       	2  & 0 &  1 &  3 & 1 &  4  & 0 \\
       	0  & 3 &  3 &  2 & 1 &  2  & 1 
       	\end{array} \right].$$
       	The spectrum of $A$ is, approximately, $\{15, -3.48\pm 0.66i,~ 2.07\pm 2.30i,~  -0.09\pm 1.10i \}$, where the Perron eigenvalue of $A$ is exactly $\l_p = 15$. The Perron eigenvector of $A$ is $v_p = (3,\, 2,\, 1,\, 1,\, 1,\, 1,\, 1)^T$. The constant row-sum matrix obtained by applying Corollary \ref{c1} to $A$ is
       	$$B\,=\, \left[ \begin{array}{r r r r r r r r}
       	2  & 2 &  2 &  3 & 2 &  2 & 2 \\
       	3  & 2 &  2 &  0 & 2 &  3 & 3 \\
       	0  & 2 &  3 &  2 & 4 &  2 & 2 \\
       	6  & 2 &  2 &  0 & 2 &  1 & 2 \\
       	3  & 4 &  1 &  3 & 0 &  3 & 1 \\
       	6  & 0 &  1 &  3 & 1 &  4 & 0 \\
       	0  & 6 &  3 &  2 & 1 &  2 & 1
       	\end{array} \right].$$
       	The matrix $B$ is similar to $A$ and therefore has the same spectrum, with $\l_p$ being equal to its constant row-sum.
       	According to Corollary \ref{c1}, all the non-Perron eigenvalues of $A$ are in the Gershgorin region of the second type of $B^T$.
       	This region can be refined by applying Theorem \ref{t5} to $B$ to obtain the matrices 
       	$$F\,=\ \left[ \begin{array}{r r r r r r r r}
       	-1  & 0 &  0 &  1 &  0 &  0 &  0  \\
       	 0  & 0 &  0 & -2 &  0 &  1 &  1  \\
       	-3  & 0 &  1 &  0 &  2 &  0 &  0  \\
       	 3  & 0 &  0 & -2 &  0 & -1 &  0  \\
       	 0  & 2 & -1 &  1 & -2 &  1 & -1  \\
       	 3  &-2 & -1 &  1 & -1 &  2 & -2  \\
       	-3  & 4 &  1 &  0 & -1 &  0 & -1
       	\end{array} \right].$$    
       	and
       	$$G\,=\ \left[ \begin{array}{r r r r r r r r}
       	-1  & 0 &  0 &  0 &  0 &  0 &  0  \\
       	 0  & 0 &  0 & -3 &  0 &  1 &  1  \\
       	-3  & 0 &  1 & -1 &  2 &  0 &  0  \\
       	 3  & 0 &  0 & -3 &  0 & -1 &  0  \\
       	 0  & 2 & -1 &  0 & -2 &  1 & -1  \\
       	 3  &-2 & -1 &  0 & -1 &  2 & -2  \\
       	-3  & 4 &  1 & -1 & -1 &  0 & -1        	
       	\end{array} \right]$$
       	According to Theorem \ref{t5}, the non-Perron eigenvalues of $A$ are also contained in the region $S$ which is given by (\ref{e3}) and reduces, for this particular example, to the disc $ \hat{D}_{F,1}(2 ,15)$ with center $2$ and radius $15$.
       	\begin{figure}[H]
       		\centering
       		\includegraphics[scale=0.40]{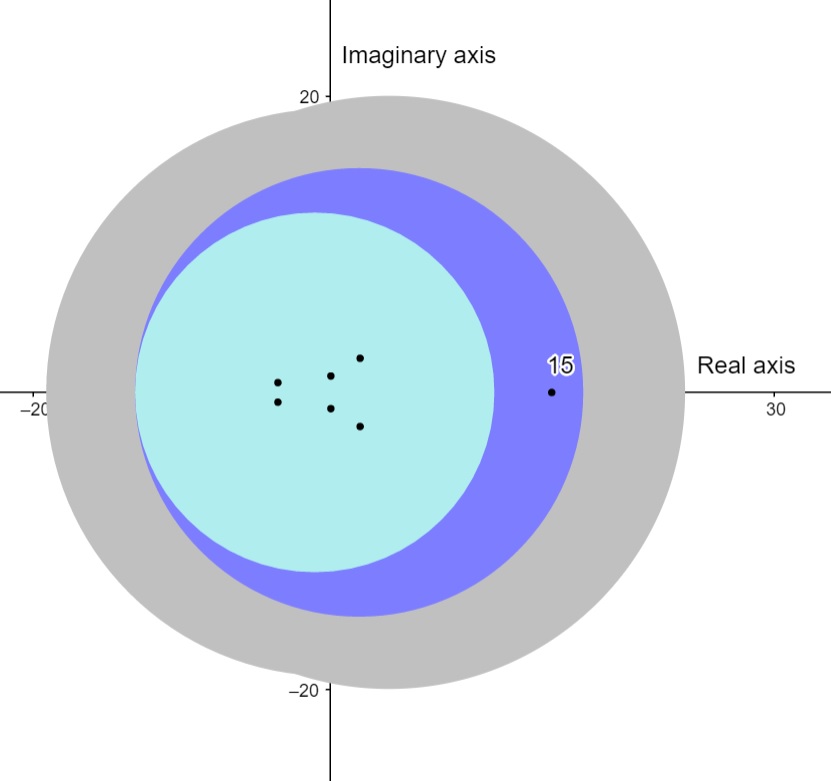}       		
       		\caption{The area in turquoise is the region given by Theorem \ref{t5}. For this particular example it is the same as the area given by Corollary \ref{c2}.
       			The \grs of $B^T$ (given by Corollary \ref{c1}) extends over the  blue and turquoise areas.
       			The original Gershgorin region of $A^T$ extends over the gray,  blue and turquoise areas.
       			The points in black represent the non-Perron eigenvalues of $A$. Observe that all of them are contained within the turquoise area given by Theorem \ref{t5}.}  
       		\end{figure}       
       \end{example}
   		%
  			\section{Second case: some components of $v$ are equal to zero}
  			Let $(\lambda, v)$ be a known real eigenpair of the $n \times n$ real matrix $A$, where $v$ has exactly $k$ components equal to $0$ for some integer $k\geq1$.
  			There is an $\n$ permutation matrix $P$ such that
  			\begin{equation}\label{e9}
  			v' = Pv = (0, \ldots, 0, v_{k+1}, \ldots,  v_n)^T.
  			\end{equation}
  			The last $(n  -  k)$ components of $v'$ are nonzero.
  			Let
  			\begin{equation}\label{e7}
  			S\,=\ \left[ \begin{array}{r r}
  			I_k  & M   \\
  			0    & I_{n-k}                      	                     	 
  			\end{array} \right],
  			\end{equation}  				  
  			where   $M$ is  $k \times  (n-k)$  with first column all $1$'s and all other entries equal to $0$.
  		    Then  $A$ is similar to the matrix   $C = SPAP^TS^{-1}$ which has  eigenvector $w =Sv'$
  		    since $Cw = (SPAP^TS^{-1}) (SPv) = \l SPv = \l w $.
  		    Now, $w  = Sv' =   (v_{k+1}, \ldots ,  v_{k+1}, v_{k+1},\ldots , v_n)^T$   has no zero entries.
  			Note that the last $(n - k)$ components of $v$ and $w$ are the same.
  			Moreover, $S$ is nonsingular and no computation is needed to find $S^{-1}$ as it is given by
  			\begin{equation}\label{e10}
  			S^{-1}\,=\ \left[ \begin{array}{r l }
  			I_k  & -M   \\
  			0    & I_{n-k}                      	                     	 
  			\end{array} \right].
  			\end{equation}
  			We arrive at the following result which can be used together with the theorems stated in the preceding sections to locate the remaining eigenvalues of $A$.
  			\begin{theorem}\label{t10}
  				Suppose that $(\lambda, v)$ is a real eigenpair of the $n \times n$ real matrix $A$, where $v$ has at least one zero component
  				and let the matrices $P$ and $S$ be as in (\ref{e9}) and (\ref{e7}).                
  				Then the matrix $C = S PAP^T S^{-1}$ is similar to $A$ and has eigenpair $(\l, SPv)$, where every component of the vector $SPv$ is real and nonzero.  				
  			\end{theorem} 			
  			
  			\bigskip
  			\begin{example}\label{ex8}
  				Let $$A\,= \left[ \begin{array}{r r r r }
  				 7  &  -10  & -2 &  2   \\
  				 5  &   -8  & -2 &  2   \\
  				-5  &   12  &  4 & -4   \\
  				-1  &    4  &  1 & -1                      	                     	 
  				\end{array} \right].$$
  				Then $0$ is an eigenvalue of $A$ with corresponding eigenvector $v= (0, 0, 1, 1)^T$.
  				The other eigenvalues of $A$ are $-1 , 1$ and $2$.
  				Let $$ S = \, \left[ \begin{array}{r r r r }
  				1  &    0 &  1 &  0   \\
  				0  &    1 &  1 &  0   \\
  				0  &    0 &  1 &  0   \\
  				0  &    0 &  0 &  1                      	                     	 
  				\end{array} \right].$$
  				Then $Sv = e$ and the matrix $A$ is similar to
  				$$ C = S A S^{-1} = \, \left[ \begin{array}{r r r r }
  				2  &    2 &  -2 & -2   \\
  				0  &    4 &  -2 & -2   \\
  				-5  &   12 &  -3 & -4   \\
  				-1  &    4 &  -2 & -1                      	                     	 
  				\end{array} \right], $$
  				which is a constant row-sum matrix since $Ce = S A S^{-1} S v = 0$.
  				To improve the location of eigenvalues, we apply Theorem \ref{t4} to $C$ to obtain the matrix
  				$$ F = \, \left[ \begin{array}{r r r r }
  				3  &  -2 &  0 &  0   \\
  				1  &   0 &  0 &  0   \\
  				-4  &   8 & -1 & -2   \\
  				0  &   0 &  0 &  1                      	                     	 
  				\end{array} \right]. $$
  				\begin{figure}[H]
  					\centering
  					\includegraphics[scale=0.40]{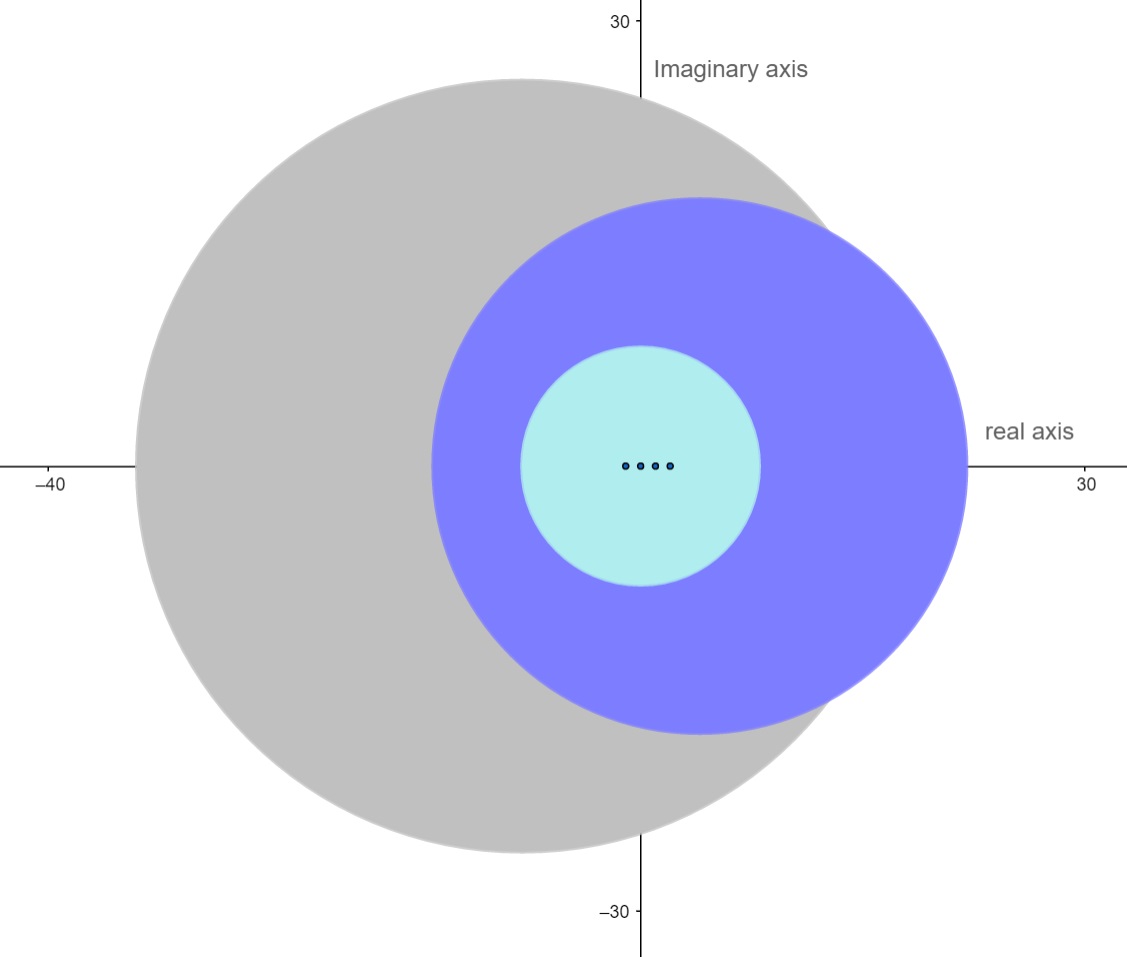}       		
  					\caption{The Gershgorin region of $A^T$ is in gray, the Gershgorin region of the second type of $C^T$ is in blue and
  						the Gershgorin region of the second type of $F^T$ is in turquoise.
  						The four small points represent the eigenvalues of $A$.}
  				\end{figure}       	
  			\end{example}
  			
  			\bigskip               
  			\begin{rem}
  				Note that Theorem \ref{t10} together with Theorem \ref{t1} actually hold for complex matrices, so that
  				every $\n$ complex matrix is similar to some constant row-sum matrix.
  				In fact, there are other ways to obtain a constant row-sum matrix that is similar to a given $\n$ complex matrix $A$. 
  				If $v$ is an eigenvector of $A$ associated with some eigenvalue $\l$,  
  				then we can find infinitely many nonsingular matrices that satisfy the equation $Sv = e$. It follows that $C = SAS^{-1}$ is a constant row-sum matrix since $ C e = C (Sv) = \l (Sv) = \l e$.
  				Some of the matrices that satisfy $Se =v$ can be easily found. However, in general, the inverse of $S$ may be costly in terms of computation which makes Theorem \ref{t10} interesting since it uses matrices $P$ and $S$ whose inverses are, respectively, $P^T$ and the simply structured matrix $S^{-1}$ given by (\ref{e10}).  
  			\end{rem}
  			%
  			\section{Comparison with other types of inclusion sets}
  			As mentioned in the introduction, the idea of using a known eigenpair to locate the spectrum of a given matrix was explored recently in \cite{M} by A. Melman. A main result in \cite{M} was achieved by combining Gershgorin's and Brauer's theorems, \cite{G} and \cite{B2}.
  			To be able to make a transparent comparison between  the location sets obtained in our work and those obtained in \cite{M}, here is a brief summary of the ideas behind the locations sets obtained in \cite{M}.\\  
  			If the spectrum of the $\n$ real matrix $A$ (counting algebraic multiplicities) is $\{ \l_1, \l_2, \dots, \l_n\}$
  			and $v = (v_1, v_2, \dots, v_n)^T$ is an eigenvector of $A$ associated with $\l_1$,
  			then by Brauer's theorem the matrix $A - vz^T$ has spectrum $\{ \l_1 - z^Tv, \l_2, \dots, \l_n\}$
  			for every  
  			$z \in \mathbb{C}^n$.
  			A main idea in $\cite{M}$ is to find a specific vector $x$ that allows for the best Gershgorin region among all matrices 
  			in $\big\{A^T - zv^T | z\in \mathbb{C}\big\}$.
  			Fortunately this is algebraically possible if the eigenpair $(\l_1,v)$ is real and there is an optimization theorem that allows for finding $x$ which is by the way a real vector in this case.
  			(For more details, look at \cite{M}.)\\
  			We remind the reader that the eigenvalues of every $\n$ complex matrix  $M = [m_{ij}]$ lie in its Ostrowski-Brauer set given by
  			\begin{equation}\label{e21}
  			\Gamma(M) = \bigcup_{\stackrel{1\leq i,j \leq n}{i<j}} \Big\{ y\in \mathbb{C}: |y-m_{ii}| |y-m_{jj}|\leq 
  			\sum_{\stackrel{1\leq k \leq n}{k\neq i}} |m_{ik}| \sum_{\stackrel{1\leq k \leq n}{k\neq j}} |m_{jk}| \Big\}.
  			\end{equation}  			
  			\noindent 			
  			It is difficult to have a complete idea on how the location sets given by Theorems \ref{t3}, \ref{t4} and \ref{t5} compare in size with those obtained in \cite{M} for every matrix. Nevertheless, we do the comparison for two matrices used in $\cite{M}$.
  			\begin{example}\label{ex12}
  			We consider the  $3\times 3$ matrices
  			$$A_1 \,=\ \left[\begin{array}{r r r r}
  			0   & 1 &  0   \\
  			2   & 5 &  4  \\
  			0   & 3 &  0    
  			\end{array}\right]
  			\text{~~~and~~~} 
  			A_2 \,=\ \left[\begin{array}{r r r r}
  			12   & 6  &   6    \\
  			3   & 3  &   18   \\
  			8   & 8  &   8    
  			\end{array}\right]
  			$$
  			which are used, respectively, in \cite[Example 1]{M} and \cite[Example 2]{M}.
  			A given eigenpair of $A_1$ consists of $\l_1 = 7$ and $v = (1, 7, 3)^T$. The other eigenvalues are $0$ and $-2$.
  			Matrix $A_2$ is constant row-sum with $A_2 e = 24$ and its other eigenvalues are $-6$ and $5$. 
  			By applying Theorem \ref{t3} then Theorem \ref{t5} to $A_1$  we obtain two matrices $F$ and $G$ the intersection of whose Gershgorin regions of the second type gives a location set for $0$ and $-2$. In a similar fashion we obtain a location set for the eigenvalues $-6$ and $5$ of $A_2$.  			
  			\begin{figure}[H]
  				\centering
  				\begin{subfigure}{0.35\textwidth}\label{fig7a} 
  					\includegraphics[scale=0.45]{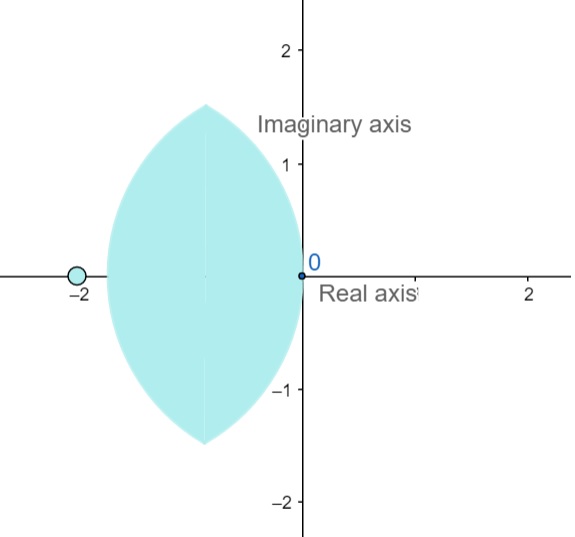} 
  					\caption{Inclusion set for the eigenvalues $0$ and $-2$ of $A_1$.%
  					}          			          			        
  				\end{subfigure}
  				\hskip20mm
  				\begin{subfigure}{0.35\textwidth}\label{fig7b}
  					\centering           			
  					\includegraphics[scale=0.30]{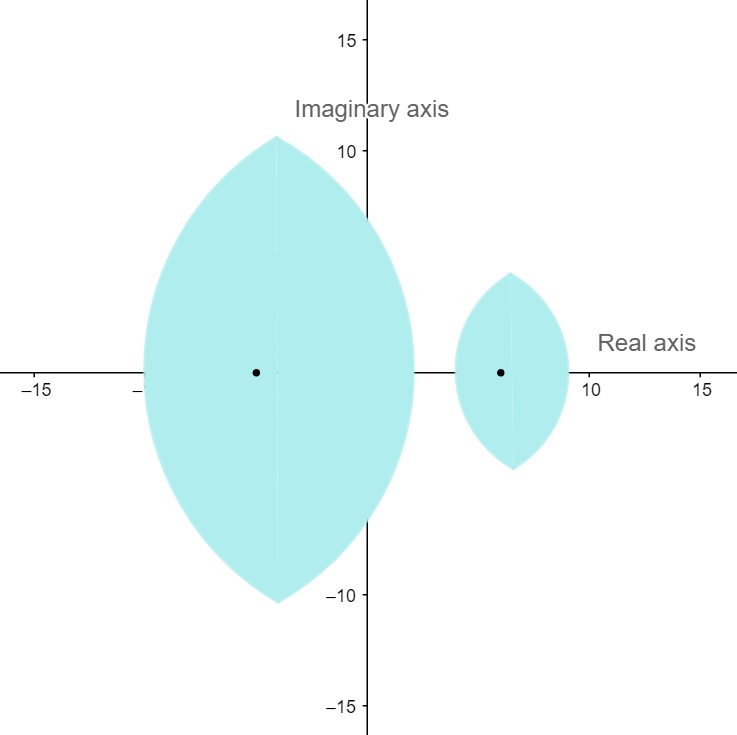}  						
  					\caption{Inclusion set for the eigenvalues $-6$ and $5$ of $A_2$}.          		
  				\end{subfigure}  
  				\caption{}
  			\end{figure}
  			
  			\noindent
  			Comparing Figure\,7\,(a) to the left hand side of \cite[Figure\,2]{M}, we can see that all location sets are close in size to each other, with the inclusion set represented in Figure\,7\,(a) being smaller  in size than the Gershgorin regions obtained in \cite[Figure\,2]{M}.
  			A similar conclusion is made when comparing Figure\,7\,(b) with the left hand side of \cite[Figure\,3]{M}.\\
  			Note that these conclusions hold for these particular matrices and cannot be considered general facts. 
  		\end{example}
  			
  			\noindent
  			For comparison related to Ostrowski-Brauer sets, we take into consideration that each of the matrices $F$ and $G$ from (Theorems \ref{t4} and \ref{t5}) has its own Ostrowski-Brauer set and their intersection can be compared to the Ostrowski-Brauer sets obtained in \cite{M}. However, we need first to show that $\l_2, \l_3, \dots, \l_n$ are eigenvalues of $F$ as well as of $G$.			
  			\begin{theorem}\label{t8}
  				Let $A$ be an $\n$ real matrix with spectrum $\{\l_1 , \l_2, \dots, \l_n\}$ (counting algebraic multiplicities).
  				Suppose that $\l_1$ is real and associated with a real eigenvector $v$ which has no zero components.
  				Let $B = S^{-1} A S$, where $S =$ diag$(v)$. Let $F$ and $G$ be the matrices obtained from $B$ in Theorems \ref{t4} and \ref{t5}. 
  				Then, for $i=2, 3, \dots, n$, we have
  				\begin{equation}\label{e22}
  				\l_i \in \Gamma(F) \cap \Gamma(F^T), \text{~~ if~~} n \text{~~is even}
  				\end{equation}
  				and
  				\begin{equation}\label{e23}
  				\l_i \in \Gamma(F) \cap \Gamma(F^T) \cap \Gamma(G) \cap \Gamma(G^T), \text{~~ if~~} n \text{~~is odd}.
  				\end{equation} 		 		
  			\end{theorem}
  			\begin{proof}
  			In the case $n$ is even, observe that  $F$ is obtained in Theorem \ref{t4} by subtracting from every column of $B$  a real multiple of the vector $e$. That is 
  			\begin{equation}
  			F = B - [\alpha_1 e ~~ \alpha_2 e ~~ \dots ~~ \alpha_n e],
  			\end{equation} 
  			where $\a = (\alpha_1, \alpha_2, \dots, \alpha_n)^T \in \mathbb{R}^n$.
  			Hence $F$ has the form $F = B - e\a^T$ and it follows by Brauer's Theorem that
  			$\l_2, \l_3, \dots, \l_n$ are eigenvalues of $F$. Hence, $\l_i \in \Gamma(F) \cap \Gamma(F^T)$ for $i=2, 3, \dots, n$.\\
  			In the case where $n$ is odd,  we have
  			$F = B - e\a^T$ and $G = B - e\beta^T$ for some $\a$ and $\beta \in \mathbb{R}^n$ (look at Theorem \ref{t5} to see how $F$ and $G$ are obtained).  			
  			It follows by Brauer's Theorem that $\l_2, \l_3, \dots, \l_n$ are eigenvalues of each of $F$ and $G$.
  			Hence, they lie in each of $\Gamma(F), \Gamma(F^T), \Gamma(G)$ and $\Gamma(G^T)$ and consequently in their intersection.  		
  			\end{proof}
  			\noindent 			
  			\begin{example}\label{ex13}
  				We look again at the matrices $A_1$ and $A_2$ in the previous example, which are also given in \cite[Example 1]{M} and \cite[Example 2]{M}.
  				For the matrix $A_1$, by using (\ref{e21}) we find that $\Gamma(F)\cap\Gamma(F^T) = \{-2 , 0\}$ which is perfect since it consists of two points only. Note that all three eigenvalues of $F$ are in this set since $0$ is an eigenvalue of $F$ with multiplicity $2$.  In fact, these eigenvalues can be obtained simply by looking at $F$:
  				$$
  				G \,=\ \left[\begin{array}{r r r r}
  				0     &  0 &  0   \\
  				2/7   & -2 &  12/7  \\
  				0     &  0 &  0    
  				\end{array}\right].
  				$$
  				For the matrix $A_2$, by using  (\ref{e21}) and Theorem \ref{t8}, we obtain the following location set.
  					\begin{figure}[H]
  					\centering
  					\includegraphics[scale=0.40]{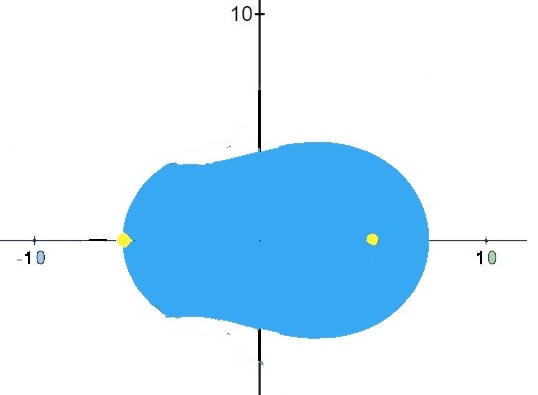}       		
  					\caption{The colored area represents $\Gamma(F) \cap \Gamma(F^T) \cap\Gamma(G) \cap \Gamma(G^T)$. The two points in yellow are the eigenvalues $-6$ and $5$ of $A_1$.}
  				 \end{figure}
  			
 					\noindent 			
  					We can see that the region obtained here is close in size to the Ostrowski-Brauer sets obtained in \cite[Figure 3]{M}. (In \cite[Figure 3]{M}, the parts of the Ostrowski-Brauer set are each enclosed within a rectangle.)\\
  					
  					\noindent
  					At the end of this section we highlight the following points.
  					\begin{enumerate}
  						\item From the algebraic point of view, the discs given by Theorems \ref{t3}, \ref{t4} and \ref{t5} have the advantage that the radii have simple explicit linear algebraic forms given in terms of the entries of the matrix $A$. (See Definition \ref{d1}.)
  						\item In terms of computations, we didn't see much difference since, for all types of discs obtained here and in \cite{M}, the calculation of the radii can be done with a number of operations of order $O(n)$.
  						\item As we mentioned earlier, we would like to emphasize that the location sets given by Theorems \ref{t3}, \ref{t4} and \ref{t5} allow for a significant reduction of the Gershgorin region of the transpose of a large matrix $A$ if this matrix is dense and the elements in each row (or in each column for $A^T$) are relatively close to each other. This fact can be clearly seen from Definition \ref{d1}.
  						\item Which location is the best depends on the matrix being studied. One may actually consider the intersection of all of them.
  					\end{enumerate}
  			\end{example} 				 	
  			\section{Bounding the largest in absolute value of the remaining eigenvalues of $A$}
  			%
  			Let $A$ be an $\n$ real matrix with known real eigenpair $(\l,v)$. In this section we are going to derive some upper bounds for the largest in absolute value of the remaining eigenvalues of $A$. The result applies, in particular, to every nonnegative irreducible matrix with known Perron eigenpair $(\l_p, v_p)$.
  			\subsection{Some upper bounds derived from the preceding location theorems}
  			It is natural that the farthest point of an inclusion set from the origin provides an upper bound for the absolute values of all eigenvalues comprised inside this region. Each of the locations obtained in the previous sections is a union of discs with centers and radii given explicitly in terms of the entries of the matrix. This allows for an easy derivation of the bounds.
  			Let $(\l_1,v)$ be a known real eigenpair of the real $\n$ matrix $A = [a_{ij}]$.
  			Without loss of generality, we assume that every components of $v$ is nonzero. If this is not the case, then we use Theorem \ref{t10}.\\	
  			By Theorem \ref{t3}, every eigenvalue $\l$ of $A$ different from $\l_1$ is in the Gershgorin region of the second type of the  matrix $B^T$ given by $B = S^{-1} A S,$ where $S= \text{diag}(v)$
  			This region is made up of the discs $\{ \hat{D}_i (a_{ii}, \hat{r}_i), ~i=1,\, 2 , \, \dots,\, n\}$, where the radius $\hat{r}_i$ is obtained from the $i$th column of $B$ according to Definition \ref{d1}.
  			Let $\l_2$ be a largest eigenvalue of $A$ in absolute value such that $\l_2 \neq \l_1$. Then there exists $i \in \{1, 2, \dots, n\}$ such that $\l_2 \in \hat{D}_i (a_{ii}, \hat{r}_i)$. That is $|\l_2-a_{ii}| \leq \hat{r}_i$. This implies that $|\l_2| \leq |a_{ii}| + \hat{r}_i$, from which we obtain the upper bound 
  			\begin{equation}
  			 |\l_2| \leq \underset{ 1\leq i \leq n}{\max} \{ |a_{ii}| + \hat{r}_i\}.
  			\end{equation}  			
  			If $n$ is even, this upper bound can be improved by considering the matrix $F = [f_{ij}]$ obtained from $B$ by using Theorem \ref{t4}.
  			Then following the same reasoning as above we obtain
  			\begin{equation}
  			|\l_2| \leq \underset{ 1\leq i \leq n}{\max} \{ |f_{ii}| + \hat{r}_i(F^T)\},
  			\end{equation}
  			where  $\hat{r}_i(F^T)$ is the radius of the \gs obtained from the $i$th column of $F$.
  			From this discussion, we state the following theorem where two cases are considered depending on whether the size of $A$ is odd or even. 	
  			
  			\begin{theorem}\label{t7}
  				Let 
  				$(\l_1 , v)$ be a known real eigenpair of the $\n$ real matrix $A = [a_{ij}]$.
  				Suppose that every component of $v$ is nonzero.
  				Let $S= \text{diag}(v)$ and $B = S^{-1} A S$.
  				If $\l$ is an eigenvalue of  $A$ different from $\l_1$, then 
  				\begin{equation}\label{e4}
  				|\l| \leq \underset{ 1\leq i \leq n}{\max} \{ |a_{ii}| + \hat{r}_i \},
  				\end{equation}  				  
  				where $\hat{r_i}$ is the radius of the \gs obtained from the $i$th column of $B$.
  				Moreover,
  				\begin{enumerate}
  					\item If $n$ is even and $F=[f_{ij}]$ is the matrix obtained from $B$ by Theorem \ref{t4}, then  				 	
  					\begin{equation}\label{e5}
  					|\l| \leq \underset{ 1\leq i \leq n}{\max} \{ |f_{ii}| + \hat{r}_i(F^T)\},
  					\end{equation}
  					where  $\hat{r}_i(F^T)$ is the radius of the \gs obtained from the $i$th column of $F$.
  					\item If $n$ is odd and $F=[f_{ij}]$ and $G=[g_{ij}]$ are the matrices obtained from $B$ by Theorem \ref{t5}, then  				 	
  					\begin{equation}\label{e6}
  					|\l| \leq \min \Big\{ \underset{ 1\leq i \leq n}{\max} \{ |f_{ii}| + \hat{r}_i(F^T)\},
  					\underset{ 1\leq i \leq n}{\max} \{ |g_{ii}| + \hat{r}_i(G^T)\}\Big\},
  					\end{equation}
  					where  $\hat{r}_i(F^T)$  and $\hat{r}_i(G^T)$ are, respectively,  the radii of the \gss obtained from the $i$th columns of $F$ and $G$.
  				\end{enumerate}
  			\end{theorem}
  		This theorem applies, in particular, to every $\n$ nonnegative irreducible matrix $A$ with known Perron eigenpair $(\l_p, v_p)$. 
  		Note that, in this case, the Perron eigenvalue $\l_p$ is itself an upper bound for $\l$. Therefore the upper bounds given by Theorem \ref{t7} are considered in the case they are smaller than $\l_p$. In other words, the upper bound given by (\ref{e4}) is better than $\l_p$ in the case where $\l_p$ is outside the \grs of $B$. The same idea applies to $F$ and $G$.
  		\begin{example}\label{ex4}
  			Let
  			$$A\,=\ \left[\begin{array}{r r r r}
  			4   & 3 &  4  &  6   \\
  			8   & 4 &  8  &  16   \\
  			16  & 8 &  2  &  12   \\
  			6   & 3 &  4  &  4 
  			\end{array} \right].$$ 
  			The Perron eigenvalue of $A$ is $\l_p = 24$ associated with Perron eigenvector $v_p = (1,\, 2,\, 2,\, 1)^T$.
  			The remaining distinct eigenvalues of $A$ are $\l_2 = -6$ and $\l_3 = -2$ ($\l_3$ has multiplicity $2$).
  			Second largest eigenvalue in absolute value $|\l_2| = 6$ can be bounded as follows.
  			We apply  Corollary \ref{c1} to $A$ to obtain the matrix 
  			$$B\,=\ \left[\begin{array}{r r r r}
  			4  & 6 &  8 &  6   \\
  			4  & 4 &  8 &  8   \\
  			8  & 8 &  2 &  6   \\
  			6  & 6 &  8 &  4 
  			\end{array} \right].$$ 
  			By Theorem \ref{t7}, the matrix $B$ gives an upper bound $m_B =14 > |\l_2| =6$.
  			To improve this bound, we first apply Theorem \ref{t4} to $B$ to obtain the matrix
  			$$F\,=\ \left[\begin{array}{r r r r}
  			-2  & ~0 &  ~0 &  ~0   \\
  			-2  & -2 &  ~0 &  ~2   \\
  			~2  & ~2 &  -6 &  ~0   \\
  			~0  & ~0 &  ~0 &  -2 
  			\end{array} \right].$$
  			Then we apply Theorem \ref{t7} to $F$ to obtain a new upper bound $m_F = 6 \geq |\l_2| = 6$.  				
  		\end{example}
  		 In the preceding example, the bound $m_F$ is perfect as it is equal to $|\l_2|$. This is not always the case.  		
  		 In fact, looking at several nonnegative matrices, we observed that the gap between  the second largest eigenvalue in absolute value and the bounds given by Theorem \ref{t7} can be sometimes relatively large. 
  		 In general, the eigenvector $v$, in Theorem \ref{t7},  can be associated with any eigenvalue of $A$ and not necessarily  with its spectral radius. This implies that the bounds given by Theorem \ref{t7} could be simply upper bounds of the spectral radius itself and therefore, should be compared to the many spectral radius upper bounds that are in the literature. In general, if $A$ has a known eigenpair $(\l,v)$, then the bounds given by Theorem \ref{t7} are bounding the largest absolute value among the eigenvalues of $A$ other than $\l$.
  		\begin{example}\label{ex5}
  			Let
  			$$
  			A\,=\ \left[ \begin{array}{r r r r r r r r}
  			-18  & 3   &  2 &  0  &  3 &  0  &   0   \\
  			12   & -18 &  0 & 12  &  0 &  0  &   12  \\
  			18   & 0   &-24 &  18 &  0 &  9  &   18  \\
  			0    & 3   &  2 & -24 &  3 &  3  &   0   \\
  			12   & 0   &  0 &  12 &-18 &  6  &   0   \\
  			0    & 0   &  4 &  12 &  6 & -24 &   12  \\
  			0    & 3   &  2 &  0  &  0 &  3  &  -18
  			\end{array} \right].$$  
  			Applying Theorem \ref{t3} to $A$, we obtain the matrix
  			$$
  			B\,=\ \left[ \begin{array}{r r r r r r r r}
  			-18  & 6   &  6 &  0   &  6 &  0  &   0   \\
  			6  & -18 &  0 &  6   &  0 &  0  &   6  \\
  			6  & 0   &-24 &  6   &  0 &  6  &   6  \\
  			0  & 6   &  6 & -24  &  6 &  6  &   0   \\
  			6  & 0   &  0 &   6  &-18 &  6  &   0   \\
  			0  & 0   &  6 &   6  &  6 & -24 &   6  \\
  			0  & 6   &  6 &   0  &  0 &  6  &  -18
  			\end{array} \right].$$  
  			Matrix $A$ is singular and has the eigenvector $v = (1,\, 2,\, 3,\, 1,\, 2,\, 2,\, 1)^T$ associated with the eigenvector $0$.
  			The other eigenvalues of $A$ are $-37.29,\, -32.49,\, -24,\\ -20.76,\, -15.51\,$ and $\,-13.95$.
  			Since $A$ and $B$ have the same spectrum, we apply Theorem \ref{t7} to $B$ to obtain an upper bound for the spectral radius of $A$ which is
  			$$|-37.29| \leq 42.$$
  			Note that the upper bounds given by the one norm and infinity norm of $B$ are both equal to $48$.
  			The upper bound given by the one norm of $A$ is $78$ and that given by the infinity norm of $A$ is $87$.  			
  		\end{example}
  	
  			The upper bounds given by Theorem \ref{t7} are considered for their simple and explicit algebraic forms and 
  			for their potential theoretical applications. %
  			For numerical applications, we think we have even more interesting bounds in the next subsection.  			
  		%
  		\subsection{Bounding the eigenvalues of real matrices by using a class of matrix semi-norms}
  		  	Some of the important upper bounds obtained for the second largest eigenvalue, in absolute value, of a nonnegative matrix in general and a stochastic matrix in particular can be found in  \cite{Kol}, \cite{RT}, \cite{Se1}, \cite{T}, \cite{Tan}, the valuable book \cite{Se} and the interesting survey \cite{IS}.
  		  	Note that an important class of bounds is obtained in \cite{RT} by using the technique of converting a nonnegative matrix to a nonnegative constant row-sum matrix via the idea in Theorem \ref{t1}; look at \cite[Pages 63-64]{RT}.
  		  	In general, if $A$ is an nonnegative irreducible matrix having a Perron eigenpair $(\l_p , v_p)$ and $S =$ diag($v_p$), then every bound $m$ that applies to stochastic matrices applies also to the matrix $A$, since the matrix $B = \frac{1}{\l_p} S^{-1} A S$  is stochastic and its spectrum is proportional to that of $A$.\\
  		  	Other bounds obtained by using the technique of matrix deflation are discussed in \cite{H}.
  		  	If $A$ is any real matrix having an eigenvector $v$ with no zero component and $S =$ diag$(v)$, then every bound $M$ that applies to the general case of real constant row-sum matrices applies also to the matrix $A$, since the matrix $B = S^{-1} A S$ is constant row-sum and has the same spectrum as $A$. Such upper bounds are discussed in our recent articles \cite{HM4} and \cite{HM5}. If the eigenvector $v$ of $A$ has some components equal to zero, then Theorem \ref{t10} can be used.\\
  		  	%
  		  	%
  		  	Let $B$ be an $\n$ real constant row sum matrix such that $Be = \l_t e$. 
  		  	Let $\tau_p(B)$ be the nonnegative matrix function defined by
  		  	\begin{equation}\label{e14}
  		  	\tau_p(B) = \max_{\stackrel{x\in \mathbb{R}^n}{\stackrel{x^Te=0}{||x||_p =1}}} ||B^T x||_p,
  		  	\end{equation}
  		  	 where $||x||_p$ is the $l_p$-norm of the vector $x$ with $p\in \mathbb{N} \cup \{\infty\}$.
  		  	If $\l$ is an eigenvalue of $B$ different from $\l_t$, then  		  	
  		  	\begin{equation}\label{e11}
  		  	 |\l| \leq \sqrt[k]{\tau_p(B^k)}, \text{~~for every~~} k \in \mathbb{N} \text{~~\cite[Corollary 2.9]{HM5}}.
  		  	\end{equation}
  		  	%
  		  	Two functions $\tau_1(B)$ and $\tau_{\infty}(B)$ are numerically applicable because of their known explicit forms (look at \cite{IS} and the included references for stochastic matrices and at \cite{HM5} for the more general case of constant row-sum matrices):
  		  	\begin{equation}
  		  	\tau_1(B) = \frac{1}{2} \underset{i,j}{\max}\sum_{k=1}^n |a_{ik} - a_{jk}|
  		  	= \l_t - \underset{i,j}{\min}\sum_{k=1}^n \min\{a_{ik} , a_{jk}\},	
  		  	\end{equation}
  		  	and 
  		  	\begin{equation}
  		  	\tau_{\infty}(B) = \r(B),
  		  	\end{equation} 
  		  	where $\r(B)$ is defined as follows.
  		  	\begin{definition}\cite[Definition 2.3]{HM4} \label{d2}
  		  		Let $A= [a_{ij}]$ be an $\n$ real matrix.\\
  		  		Let $b_{1j}\geq \dots \geq b_{nj}$ be an arrangement in non-increasing order of 
  		  		$a_{1j}, \dots, a_{nj}$.
  		  		Then, for $j=1, \dots, n,~$ define $cs_j(A)$ in the following manner:\\
  		  		\hskip10mm $(1) ~~~ cs_j(A) =
  		  		\big(b_{1j}+\dots+b_{\frac{n-1}{2},j}\big)-\big(b_{\frac{n+3}{2},j}+\dots+b_{nj}\big),~~$
  		  		if $n$ is odd.\\
  		  		\hskip10mm $(2) ~~~ cs_j(A) = 
  		  		\big(b_{1j}+\dots+b_{\frac{n}{2},j}\big)-\big(b_{1+\frac{n}{2},j}+\dots+b_{nj}\big),~~$
  		  		if $n$ is even.\\
  		  		We define $\r(A)$ by:
  		  		$$ \r(A) = \underset{1\leq j \leq n}{\text{max}}\{cs_j(A)\}.$$
  		  	\end{definition}
  			 Counting multiplicities, let the spectrum of the real constant row-sum matrix $B$ be 
  			 $\{\l_1, \l_2, \dots, \l_n \}$, with $Be = \l_1 e$ and $|\l_2| \leq |\l_3| \leq \dots \leq |\l_n|$.
  			 We have shown, in \cite{HM5}, that
  			\begin{equation}\label{e12}
  			 \lim_{k \to \infty}~ \big(\tau_p(B^k)\big)^{1/k}=|\lambda_n|. \text{~~\cite[Theorem 2.15]{HM5}}
  			\end{equation}
  			 This convergence equation ensures good bounds for $|\l_n|$ by trying high powers of $B$.
  			 Note that if the row-sum constant $\l_1$ of $B$ is non-simple and equal in absolute value to the spectral radius of $A$, then
  			 $\lim_{k \to \infty}~ \big(\tau_p(B^k)\big)^{1/k} = |\l_1|$.\\
  			 For the case of the real matrix with a given real eigenpair $(\l_1, v)$ such that $v$ has no zero components, we have the following corollary which is an immediate consequence of (\ref{e11}).
  			 \begin{cor}\label{c3}
  			 	Let $A$ be an $\n$ real matrix having spectrum $\{\l_1, \l_2, \dots, \l_n\}$ (counting multiplicities).
  			 	Suppose that $\l_1$ is real and associated with a real eigenvector $v$ with no zero components and let $D =$ diag$(v)$.
  			 	Then for $i=2, 3, \dots, n$  and for every $k \in \mathbb{N}$ we have 
  			 	\begin{equation}\label{e18}
  			 	 |\l_i| \leq \sqrt[k]{\tau_p(B^k)},
  			 	\end{equation}
  			 	where $B$ is the constant row-sum matrix given by $B=D^{-1}AD$.
  			 \end{cor}
  		 	 
  			 \begin{remark}
  			 If some of the components of $v$ are equal to zero, then we can use Theorem \ref{t10} to obtain matrix $C = SPAP^{-1}S^{-1}$.
  			 If $C$ is constant row-sum, we apply (\ref{e18}) to it. If not, then $C$ has an eigenvector $v' = SPv$ with no zero components, which allows us to obtain a constant row-sum matrix $B$ by Theorem \ref{t1}. Then (\ref{e18}) is applied to $B$.  				
  			 \end{remark}  		
  		 \noindent
  		 The inequality in (\ref{e18}) provides an upper bound on the determinants of real matrices.
  		 \begin{cor}\label{c4}
  		 	Let $A$ be an $\n$ real matrix having real eigenpair $(\l, v)$.
  		 	Suppose that $v$ has no zero components and let $D =$ diag$(v)$.
  		 	Then for every $k \in \mathbb{N}$ we have 
  		 	\begin{equation}\label{e24}
  		 	|\text{det}(A)| \leq |\l|~\big[\tau_p(D^{-1} A^k D)\big]^{\frac{n-1}{k}}.
  		 	\end{equation}  		 	 		
  		 	In particular,
  		 	\begin{equation}\label{e25}
  		 	|\text{det}(A)| \leq |\l|~\tau_p(D^{-1} A^{n-1} D).
  		 	\end{equation}
  		 \end{cor}
  			 
  			 \begin{proposition}\label{p1}
  			 	Let $A$ and $B$ be as in Corollary \ref{c3} and let $F$ and $G$ be the $\n$ constant row-sum matrices obtained from $B$ by Theorem \ref{t4} or Theorem \ref{t5}. Then
  			 	\begin{equation}\label{e19}
  			 	\tau_p(B) = \tau_p(F), \text{~~if~~} n \text{~~is even}
  			 	\end{equation} 
  			 	and
  			 	\begin{equation}\label{e20}
  			 	\tau_p(B) = \tau_p(F) = \tau_p(G) , \text{~~if~~} n \text{~~is odd}.
  			 	\end{equation}
  			 \end{proposition}
  		 	 \begin{proof}
  		 		Observe that $F$ and $G$ are constructed by adding to each column of $B$ a multiple of the all $1$'s vector $e$. That is, $F = B ~ + ~ [\a_1 e ~ \a_2 e ~ \dots \a_n e],$\\ $\a_i \in \mathbb{R}$ and $G$ has a similar form.
  		 		Then (\ref{e19}) and (\ref{e20}) follow from the definition of $\tau_p$ given by (\ref{e14}).
  		 	 \end{proof}
  			 \begin{rem}
  			 	 Proposition \ref{p1} may be considered for computation purposes. In the following examples, $F$ has a simpler structure compared to $B$ since it contains more zeros and smaller integers in absolute value. That makes the calculation of $\tau_1(F)$ and $\tau_{\infty}(F)$ faster. At the end, they have the same values as $\tau_1(B)$ and $\tau_{\infty}(B)$.
  			 \end{rem}
  			 \begin{example}\label{ex6}
  			 	In the case where $v$ has no zero component, we reconsider matrix $A$ in Example \ref{ex4}. 
  			 	The given eigenpair consists of $\l_p = 24 $ and $v_p = (1,\, 2,\, 2,\, 1)^T$.
  			 	By applying (\ref{e18}) to $F$, we have the following bounds for $|\l_2| = 6$. 
  			 	\begin{center}
  			 		\begin{tabular}{|c|c|c|c|}\hline
  			 			$F^k$         	 						&   $F$   &  $F^3$  &  $F^{3}$     \\ \hline
  			 			$\sqrt[k]{\tau_{\infty}(F^k)}\approx$ 	&   $6$   &  $6$    &  $6$          \\ \hline
  			 			$\sqrt[k]{\tau_1(F^k)} \approx$ 		&   $8$   &  $6.93$ &  $6.54$       \\ \hline
  			 		\end{tabular}
  			 	\end{center}		
  		 		Note that the bound $\tau_{\infty}(F)$ is equal to the bound obtained in Example \ref{ex4} by applying Theorem \ref{t7} to $F$.
  			 \end{example}
  			 %
  			 \begin{example}\label{ex7}
  			 	For the case where $v$ has some zero components, we go back to Example \ref{ex8}.
  			 	The given eigenvalue is $\l_1 =0$ associated with eigenvector $v= (0, 0, 1, 1)^T$.
  			 	The eigenvalue we need to bound is $\l_2 = 2$.  			 	 				 
  			 	Some bounds obtained by applying (\ref{e18}) to $F$ are given in the following table.
  			 	\begin{center}
  			 		\begin{tabular}{|c|c|c|c|}\hline
  			 			$F^k$         	 		                &   $F$   &  $F^2$   &  $F^5$      \\ \hline
  			 			$\sqrt[k]{\tau_{\infty}(F^k)}\approx$ 	&   $8$   &  $3.16$  &  $2.51$     \\ \hline
  			 			$\sqrt[k]{\tau_1(F^k)} \approx$ 	    &   $10$  &  $3$     &  $2.34$     \\ \hline
  			 		\end{tabular}
  			 	\end{center}
  		 	Application of Theorem \ref{t7} to $F$ gives a bound $m_F =10$. This is the same as $\tau_1(F)$. However,		
  		 	the bound $\sqrt[5]{\tau_1(F^5)} \approx 2.34$ is much closer to $|\l_2| = 2$. 
  			 \end{example}
  		\bigskip
  		 \subsection{Comparison to some existing bounds}
  		 Finally, we make  comparisons between the bounds discussed in this section and those obtained in \cite{M}.
  		 For that, we consider the matrix given in \cite[Example 2]{M},
  		 $$A\,=\ \left[\begin{array}{r r r r}
  		 12  & 6  &  6    \\
  		 3   & 3  &  18  \\
  		 8   & 8  &  8    
  		 \end{array} \right].$$ 
  		 The given eigenpair is $(24, e)$ and the other eigenvalues of $A$ are $\l_2 = -6$ and $\l_3 = 5$.
  		 The eigenvalue to bound is $\l_2$. 	
  		 Two bounds obtained for this matrix in \cite{M} are $m_1= 12$ and $m_2 \approx 11.36$.
  		 Hoffman's bound also was calculated for this matrix and found to be $m_3 =12$.
  		 Note that $A$ is a constant row-sum matrix and the matrices $F$ and $G$ obtained from it according to Theorem \ref{t5} are
  		  $$F \,=\ \left[\begin{array}{r r r r}
  		   4   & -2  &  -12 \\
  		  -5   & -5  &   0  \\
  		   0   &  0  &  -10    
  		  \end{array}\right]
  		 \text{~~~and~~~} 
  		   G \,=\ \left[\begin{array}{r r r r}
  		   9   &  0  &   0    \\
  		   0   & -3  &   12 \\
  		   5   &  2  &   2    
  		 \end{array}\right].
  		 $$
  		 The bound obtained by applying Theorem \ref{t7} is $m_4 = 14$.
  		  Application of (\ref{e18}) to $G$ gives bounds
  		  $m_5 = \tau_{\infty}(G) = 12$ and $m_6 = \sqrt{\tau_{\infty}(G^2)} \approx 7.75$ by the use of $\tau_{\infty}$.
  		  Using matrix semi-norm $\tau_1$, we obtain
  		  $m_7 = \tau_1(G) = 6 = |\l_2|$ and no computation is needed for second or third power of the matrix $G$. 
  		  We observe that for this particular example, the bound given by $\tau_1(G)$ is perfect and the bound given by $\tau_{\infty}$ is relatively good if applied to $G^2$.\\
  		  In general, at the cost of a some matrix power computations, the bounds given by Corollary \ref{c3}
  		  outperform any other bound since their convergence is ensured by (\ref{e12}).
  		  %
  		  %
  		  %
  		  \section{Conclusion}
  		  In this work we have shown how the location of eigenvalues can be improved if a real eigenpair of the real matrix is known.
  		  A significant improvement of the original Gershgorin region of the transpose of the matrix is obtained if this matrix is large, dense and the elements in each row (or in each column if $A^T$ is considered) are close to each other. 
  		  The ideas being discussed give rise to some questions such as:\\
  		  \begin{enumerate}
  		  \item Are we able to find similar locations in the more general case of complex matrices?
		  \item If two or more independent eigenvectors are known to be associated with the same eigenvalue, then how does this affect the location of the other eigenvalues of the matrix?
  		  \end{enumerate}
  		  %
          %
   	      %
     
\end{document}